\documentclass[11pt,leqno]{amsart}
\usepackage[letterpaper, total={6in, 8in}, left=1.25in, top=1.5in]{geometry}
\usepackage{amsmath}
\usepackage{amsfonts}
\usepackage{amssymb}
\usepackage{graphicx}
\usepackage{color}
\usepackage{hyperref}
\usepackage{xcolor}

\newtheorem{theorem}{Theorem}[section]
\newtheorem*{theorem*}{Theorem}
\newtheorem{lemma}{Lemma}[section]
\newtheorem{corollary}[theorem]{Corollary}
\newtheorem{proposition}{Proposition}[section]

\def\a{\alpha}
\def\l{\lambda}

\def\p{\partial}

\def\R{\mathbb{R}}

\def\div{\operatorname{div}}

\def \p {\partial}

\numberwithin{equation}{section}

\begin{document}

\title[Robin eigenvaule of the Weighted Laplacian]{An isoperimetric inequality  for the second Robin eigenvalue of the Weighted Laplacian}

\author{Yi Gao} 
\address{School of Mathematical Sciences, Soochow University, Suzhou, 215006, China}
\email{yigao\underline{~}1@163.com}

\author{Kui Wang} 
\address{School of Mathematical Sciences, Soochow University, Suzhou, 215006, China}
\email{kuiwang@suda.edu.cn}

\author{Anqiang Zhu} 
\address{School of Mathematics and Statistics, Wuhan University, Wuhan 430072, China}
\email{aqzhu.math@whu.edu.cn}

\subjclass[2010]{35P15, 58G25}

\keywords{Second Robin eigenvalue, Weighted measure, Isoperimetric inequality, Steklov eigenvalue}

\begin{abstract}
In this paper, we investigate a shape optimization problem for the  second Robin eigenvalue of the weighted Laplacian on bounded Lipschitz domains symmetric about the origin. Our main theorem states that the ball centered at the origin maximizes the second Robin eigenvalue among all Lipschitz bounded domains of prescribed weighted measure and symmetric about the origin for a range of negative Robin parameters.

\end{abstract}
\maketitle

\section{Introduction}
The classical Szeg\"o-Weinberger inequality \cite{Sz54, Wei56} states that the ball uniquely maximizes the second (first nonzero) Neumann eigenvalue of Laplacian  among all bounded domains of the same volume in Euclidean space.   Using the arguments of  Szeg\"o and Weinberger, analogous inequalities for the second Neumann eigenvalues have been established in various settings:  Ashbaugh-Benguria \cite{AB95} extended the  the inequality to bounded domains in hemisphere and in hyperbolic space; Chiacchio and di Blasio \cite{CB12} obtained a sharp isopermetric inequality for the first nonzero Neumann eigenvalue for bounded domians symmetric about the origin in Gauss space. 
Moreover, it is expected (cf. \cite[Problem 4.41]{Hen17}) that for a range of negative Robin parameters, the second Robin eigenvalue should be maximized by  a round ball. This expectation has recently been confirmed by Freitas and Laugesen via the Szeg\"o \cite{FL20} and the Weinberger  \cite{FL21} approaches. 
For further results on Szeg\"o-Weinberger type inequalities, we refer to \cite{AS96,  As99, BBC20, BCB16,  LL23, LWW22, LWW23, MW24, Wang19, XW23} and references therein.

Eigenvalue optimization problems for the weighted Laplacian have also attracted  considerable attention. For instance,  Chiacchio and di Blasio \cite{CB12} proved that, among all origin-symmetric regions with  fixed Gaussian volume, the ball maximizes the first nonzero Neumann eigenvalue; Brock, Francesco and di Blasio \cite{BCB16} obtained  optimal  Szeg\"o-Weinberger type inequalities under a weighted measure.
Additional results on weighted eigenvalue problems can be found in \cite{BCHT13, BCKT16, BCB16, CM24, CG22, GW26} and the references therein.

Inspired by the work of Brock, Francesco and di Blasio \cite{BCB16} on the weighted Laplace operator and that of Freitas and Laugesen  \cite{FL21} on the second Robin eigenvalue, we study the shape optimization problem for the second Robin eigenvalue of the following class of problems
\begin{align}\label{eq 1.1}
    \begin{cases}
        -\div (e^{h(|x|)}\nabla u)=\l e^{h(|x|)} u, \quad & x\in \Omega,\\
        \frac{\p u}{\p \nu}+\a u=0, \quad & x\in \p \Omega,
    \end{cases}
\end{align}
where $\a\in \R$. Here and in the sequel, $\Omega$ denotes a bounded connected  domain in $\R^m$ with Lipchitz boundary, $\nu$ is the outward normal to $\p \Omega$, and $h(r)$ is a $C^2$ function on $[0, +\infty)$.  We define the weighted measure associated with $h$ by 
$$
d\gamma_h=e^{h(|x|)}\,dx, \qquad  x\in \R^m,
$$
and denote by $W^{1,2}(\Omega; \gamma_h)$   the weighted Sobolev space equipped with the norm
\begin{align*}
||u||_{W^{1,2}(\Omega;\gamma_h)}:=\left(\int_\Omega u^2 \, d\gamma_h\right)^{1/2}+ \left(\int_\Omega |\nabla u|^2 \, d\gamma_h\right)^{1/2}. 
\end{align*}

It is well known that the problem \eqref{eq 1.1} has discrete  spectrum. We  denote its eigenvalue by $\l_{k, \a}(\Omega; \gamma_h)$ for $k=1,2, \cdots$, which satisfy
  \begin{align*}
       \l_{1, \a}(\Omega;\gamma_h) < \l_{2, \a}(\Omega;\gamma_h)  \leq \l_{3, \a}(\Omega;\gamma_h) \leq \cdots \to +\infty,
    \end{align*}
with each eigenvalue repeated according to its multiplicity. 
By standard spectral theory for self-adjoint compact operators,  the first two eigenvalues of \eqref{eq 1.1} admit the variational characterizations: 
\begin{align}
     \lambda_{1,\alpha}(\Omega; \gamma_h)=\inf\{\frac{\int_{\Omega}|\nabla u|^{2}d\gamma_h+\alpha\int_{\partial \Omega}u^{2} e^{h}\, dA}{\int_{\Omega}u^{2}d\gamma_h}: u\in W^{1,2}(\Omega;\gamma_h)\setminus\{0\}\},
\end{align}
where $dA$ is  the
induced measure on $\p \Omega$;
and
\begin{align}\label{eq 1.3}
    \lambda_{2,\alpha}(\Omega;\gamma_h)=\inf\{\frac{\int_{\Omega}|\nabla u|^{2}d\gamma_h+\alpha\int_{\partial \Omega}u^{2} e^{h}dA}{\int_{\Omega}u^{2}d\gamma_h}: u\in W^{1,2}(\Omega;\gamma_h)\setminus\{0\}, \int_{\Omega}uu_{1}d\gamma_h=0\},
\end{align}
where $u_1(x)$ is an eigenfunction corresponding to $\l_{1, \a}(\Omega;\gamma_h)$. 

To state our main results, we introduce the  following Steklov  eigenvalue problem on $\Omega$ 
\begin{align}\label{eq 1.4}
\begin{cases}
      -\div (e^{h(|x|)}\nabla u)=0 \quad & x\in \Omega,\\
        \frac{\p u}{\p \nu}=\sigma u, \quad & x\in \p \Omega.
\end{cases}    
\end{align}
The first nonzero Steklov eigenvalue can be characterized  variationally by
\begin{align}
\sigma_1(\Omega; \gamma_h)=\inf\left\{ \frac{\int_{\Omega} |\nabla u|^2 \, d\gamma_h}{\int_{\partial \Omega}  u^2 e^{h}\, dA} : u\in W^{1,2}(\Omega; \gamma_h )\setminus\{0\},\;\; \int_{\partial \Omega} ue^{h(|x|)}\ dA=0 \right\}.
\end{align}
 In this paper,  we establish the following optimal isoperimetric inequality for the second Robin eigenvalue of \eqref{eq 1.1}.
\begin{theorem}\label{thm 1}
Let $\Omega\subset \R^m$ be a bounded  Lipschitz domain symmetric about the origin. Suppose $h(r)\in C^2([0, +\infty))$ satisfies  $h'(r)\ge 0$ and $h''(r)\ge 0$ for $r\in (0, +\infty)$. Let  $B \subset \R^m$  be the  origin-centered  round ball with the same $\gamma_h$-volume as $\Omega$, namely $ \int_{\Omega} e^{h(|x|)}\, dx= \int_{B} e^{h(|x|)}\, dx$. Denote by $\sigma_1(B; \gamma_h)$ the first nonzero Steklov eigenvalue of \eqref{eq 1.4} on $B$, and $\l_{2,\a}(\Omega; \gamma_h)$ the second Robin eigenvalue of \eqref{eq 1.1} on $\Omega$.  If  $\a \in [-\sigma_1(B; \gamma_h), 0]$, then 
\begin{align}\label{eq 1.6}
\l_{2,\a}(\Omega; \gamma_h) \le \l_{2,\a}(B; \gamma_h).
\end{align}
Equality holds  if and only if 
$\Omega=B$. 
\end{theorem}
By taking $\a=0$, Theorem \ref{thm 1} implies that the ball centered at the
origin is the unique set maximizer the second Neumann eigenvalue $\mu_2(\Omega;\gamma_h)$ among all bounded Lipschitz domains $\Omega$ in $\R^m$ with prescribed $\gamma_h$-measure and symmetric about the origin. This result was previously  proved by Brock, Chiacchio, and di Blasio (Theorem 1.2 of \cite{BCB16}). Note that Theorem \ref{thm 1} applies in particular  to $\l_{2, \a}(\Omega; e^{|x|^2/2})$. Hence, our result provides information about  eigenvalues of the problem
$$
-\Delta u-x\cdot \nabla u=\l u
$$
which has been widely studied in literature (see, e.g. \cite{BCM12, BMP13}). 
The symmetry assumption on $\Omega$ in Theorem \ref{thm 1}
 ensures the validity of  the orthogonality conditions \eqref{eq 3.5}, as the weighted measure $e^{h(|x|)}$ is  radially
 symmetric about the origin.
It remains an interesting question that whether \eqref{eq 1.6} 
holds in the case of the Gauss measure $e^{-|x|^2/2}dx$.   

By taking $\a=-\sigma_1(B;\gamma_h)$, Theorem \ref{thm 1} yields the following Brock–Weinstock inequality, which was proved by Brock and Chiacchio \cite[Theorems 1.4 and 1.5]{BC25} and by Mao and Zhang \cite[Corollary 1.2]{MZ24}. 
\begin{corollary}\label{cor}
Under the hypotheses of Theorem \ref{thm 1}, we have 
\begin{align*}
\sigma_1(\Omega; \gamma_h)\le \sigma_1(B; \gamma_h).
\end{align*}
Equality holds if and only if $\Omega$ is isometric to the round ball $B$.
\end{corollary}

The main technique in the proof of Theorem \ref{thm 1} relies on Weinberger’s trick, in which  test functions are constructed using eigenfunctions on a round ball to estimate eigenvalues. The main difficulties in the proof include analyzing the properties of the second eigenfunction of \eqref{eq 1.1} on the round ball  and establishing the required monotonicity to apply Weinberger’s method. The rest of the paper is organized as follows.
In Section \ref{sect2}, we study properties of the Robin and Steklov eigenvalue  problems for balls in the weighted space. In Section  \ref{sect3}, we prove Theorem \ref{thm 1} and Corollary \ref{cor}.

\section{Eigenvalue Problem for Balls In the Weighted Space}\label{sect2}
In this section, we  prove several properties of the eigenvalue problems \eqref{eq 1.1} and  \eqref{eq 1.4}  for round balls centered at the origin. Denote by $B(R)\subset\R^m$ the origin-centered  ball of radius $R$. Then Steklov eigenvalue problem of \eqref{eq 1.4} on $B(R)$ takes the form 
\begin{align}\label{eq 2.1}
\begin{cases}
      -\Delta u-\nabla h\cdot \nabla u=0 \quad & \text{in $B(R)$,}\\
\frac{\p u}{\p \nu}=\sigma u \quad & \text{on $\p B(R)$,}
\end{cases}
\end{align}
and the Robin eigenvalue problem  of \eqref{eq 1.1} on $B(R)$ becomes  
\begin{align}\label{eq 2.2}
\begin{cases}
    -\Delta u-\nabla h \cdot \nabla u=\l u \quad & \text{in $B(R)$,}\\
\frac{\p u}{\p \nu}+\a u=0 \quad & \text{on $\p B(R)$,}
\end{cases}
\end{align}
where $\nu$ is the unit outer normal  on $\p B(R)$.

It is well known that the first nonzero Steklov eigenvalue of  the Laplacian on $B(R)$ is give by $1/R$. For the weighted Steklov eigenvalue problem \eqref{eq 1.4}, we obtain the following upper bound for $\sigma_1(B(R); \gamma_h)$.
\begin{proposition}
      Let $\sigma_1(B(R); \gamma_h)$ be the first nonzero Steklov eigenvalue of  \eqref{eq 2.1}. If $h'(r)\ge 0$, then 
 \begin{align}\label{eq 2.3}
\sigma_1(B(R); \gamma_h)\le \frac 1 R.
 \end{align}   
\end{proposition}
\begin{proof}
Recall that the first nonzero eigenvalue of \eqref{eq 2.1} is characterized by the  Rayleigh quotient
\begin{align*}
\sigma_1(B(R);\gamma_h)=\inf\left\{ \frac{\int_{B(R)} |\nabla u|^2 \, d\gamma_h}{\int_{\partial B(R)}  u^2 e^{h(|x|)}\, dA} : u\in W^{1,2}(B(R))\setminus\{0\},\;\; \int_{\partial B(R)} ue^{h(|x|)}\ dA=0 \right\}.
\end{align*}
For each $1\le i\le m$, we have 
\begin{align*}
\sigma_1(B(R); \gamma_h) \int_{\partial B(R) }  x_i^2 e^{h(|x|)}\, dA   \le  \int_{B(R)} |\nabla x_i|^2 \, d\gamma_h
\end{align*}
 Summing  over $i$ yields
\begin{align*}
  \sigma_1(B(R); \gamma_h)  \le  \frac{\int_{B(R)} \sum_{i=1}^m|\nabla x_i|^2\, d\gamma_h}{\int_{\partial B(R) }  \sum_{i=1}^m x_i^2 e^{h(|x|)}\, dA}=\frac{m\int_0^R e^{h(r)} r^{m-1}\, dr}{e^{h(R)}R^{m+1}}\le\frac{m\int_0^R e^{h(R)} r^{m-1}\, dr}{e^{h(R)}R^{m+1}}=\frac 1 R,
\end{align*}
where we used the $h'\ge 0$ in the second inequality. 
\end{proof}
We now use the  separation of variables to study the Robin eigenvalue problem \eqref{eq 2.2}. Let
$$
u(r, \theta) = w(r) T(\theta), \quad r\in (0, R), \quad \theta\in \mathbb{S}^{m-1},
$$
be a Robin eigenfunction on $B(R)$. Substituting into  \eqref{eq 2.2}  gives
\begin{align*}
-w''(r)-(\frac {m-1} r+h'(r)) w'(r)-\frac{\Delta_{S_r}T(\theta)}{T(\theta)} w(r)=\l w(r), 
\end{align*}
which implies that the spherical part $T(\theta)$ is an eigenfunction of Laplace–Beltrami operator on the distance sphere $S_r$.
Since $\l_{1, \a}(B(R); \gamma_h)$ is simple and $B(R)$ is rotationally symmetric, the first Robin eigenfunction of \eqref{eq 2.2} must be radial. Hence, $\l_{1, \a}(B(R);\gamma_h)$ is the first eigenvalue of radial problem
\begin{equation} \label{eq 2.4}
\begin{cases}
w'' + \left( \dfrac{m-1}{r} + h'(r) \right)w' +\l w=0 & \text{in } (0,R), \\
w'(0) =0, \qquad  w'(R)+\alpha w(R) = 0.
\end{cases}
\end{equation}
The second eigenvalue $\l_{2, \a}(B(R);\gamma_h)$ of \eqref{eq 2.2} is  either the second eigenvalue of \eqref{eq 2.4} or the first eigenvalue of the problem
\begin{equation} \label{eq 2.5}
\begin{cases}
w'' + \left( \dfrac{m-1}{r} + h'(r) \right)w' + (\l-\frac {m-1} {r^2}) w = 0 & \text{in } (0,R), \\
w(0)=0, \qquad  w'(R)+\alpha w(R) = 0.
\end{cases}
\end{equation}
Let $\mu_1$ denote the first eigenvalue of \eqref{eq 2.5} and let
$g(r)$ be a corresponding eigenfunction. Then $g$ satisfies
\begin{align}\label{eq 2.7}
   g'' + \left( \dfrac{m-1}{r} + h' \right)g' + (\mu_1-\frac {m-1} {r^2}) g = 0 
\end{align}
for $r\in (0, R)$, with  $g(0)=g'(R)+\a g(R)=0$. Moreover, 
$\mu_1$ admits the variational characterization
\begin{align}\label{eq 2.8}
 \mu_1= \inf_{ g\in W^{1,2}((0,R))} \left\{
    \frac{\int_0^R\left( g'(r)^2+\frac{m-1}{r^2} g^2\right) r^{m-1}e^{h(r)} \, dr +\a  g^2(R)R^{m-1}e^{h(R)}}{\int_0^R g^2 r^{m-1}e^{h(r)}\,  dr}: g(0)=0\right\},
\end{align}
and the associated eigenfunction $g$ can be chosen positive on $(0, R]$. In the following, we assume that the first eigenfunction $g(r)$ of \eqref{eq 2.5}  is positive on $(0, R)$ and satisfies $g'(0)=1$.
\begin{proposition}\label{prop 2.2}
Let $g(r)$ be the  positive  eigenfunction of \eqref{eq 2.5}. If $\a\le 0$, then
\begin{enumerate}
\item $g'(r)>0$ for $r\in (0, R)$.
\item Assume further that $\a\ge -\frac 2 R$, then  $g'(r)\ge  -\a g(r)$ for $r\in (0, R]$.
\end{enumerate}
\end{proposition}
\begin{proof}
(1) Set $$N(r)=r^{m-1}e^{h(r)}g'(r), \quad r\in[0, R].$$ 
Using \eqref{eq 2.5}, we obtain
\begin{align*}
N'(r)=\left(\frac{ m-1}{r^2}-\mu_1\right)r^{m-1}e^{h(r)} g(r).
\end{align*}
Hence $N'(r)$ has at most one zero in $(0, R]$ and is positive near $0$. Since $N(0)=0$, we have $N(r)>0$ for sufficiently small $r$. If there exists $r\in(0, R)$ such that
$N(r)<0$, then there are points $\xi_{1},\xi_{2}\in (0,R)$ with $N(\xi_{i})=0,i=1,2$. By Rolle's theorem, there exist $\eta_{1}\in (0,\xi_{1})$ and $\eta_{2}\in (\xi_{1},\xi_{2})$ such that $N'(\eta_{i})=0,i=1,2$, a contradiction. Therefore, $N(r)\geq 0$ for all $r\in (0, R)$.  If $N(r_{0})=0$ for some $r_0\in (0,R)$, then $N'(r_{0})=0$, implying  $\frac{m-1}{r_{0}^{2}}=\mu_1$. Consequently,  $N^{'}(r)<0$ for $r>r_{0}$, a contradiction. Thus $N(r)>0$ for all $r\in (0,R)$, and hence $g'(r)>0$ for $r\in (0, R)$.

(2) To prove (2), set $v(r)=g'(r)/g(r)$. From \eqref{eq 2.5} we obtain
\begin{align}\label{eq 2.9}
v'+v^2+(\frac{m-1}{r}+h'(r))v+\left(\mu_1-\frac{m-1}{r^2}\right)=0,
\end{align}
with
\begin{align}\label{eq 2.10}
  v(R)=-\a,\quad v(r)>0 \quad\text{for}\quad  r\in(0,R], \quad \text{and} \quad \lim\limits_{r\rightarrow 0^+}v(r)=+\infty.  
\end{align}
We now prove that $v(r)\geq-\a$  on $(0, R]$. Assume by contradiction that there exists $r_0\in(0,R)$ such that
\begin{align*}
v'(r_0)=0,\text{\quad \quad} v''(r_0)\ge 0,\text{\quad and\quad} 0<v(r_0)<-\a.
\end{align*}
Differentiating \eqref{eq 2.9}  and using \eqref{eq 2.10}, we obtain at $r=r_0$ that
\begin{align}\label{eq 2.101}
\begin{split}
0&=v''(r_0)+(h''(r_0)-\frac{m-1}{r_0^2})v(r_0)+\frac{2(m-1)}{r_0^3}\\
&\ge (h''(r_0)-\frac{m-1}{r_0^2})v(r_0)+\frac{2(m-1)}{r_0^3}\\
&\ge -\frac{m-1}{r_0^2} v(r_0)+\frac{2(m-1)}{r_0^3},
\end{split}
\end{align}
where we used $h''\ge 0$ in the last inequality. Thus it follows from \eqref{eq 2.101} that
\begin{align*}
v(r_0)\ge \frac 2 {r_0}>\frac 2 R.
\end{align*}
Hence
$$
-\a>v(r_0)>\frac 2 R,
$$
contradicting the assumption $\a\ge -\frac 2 R$. Therefore, (2) is proved.
\end{proof}
We now establish the following properties of the second Robin eigenfunctions on round balls when $\a\le 0$.
\begin{proposition}\label{prop 2.3}
Suppose $\a\le 0$ and $h'(r)>0$. Then the second Robin eigenfunctions of \eqref{eq 2.2} are given by 
\begin{align*}
u_i(x)=g(r) \frac{x_i}{r},\quad  i=1,2, \cdots,m,
\end{align*}
where $g(r):[0, R]\rightarrow [0,\infty)$ solves  
\begin{equation}\label{eq 2.11}
g''(r)+(\frac {m-1} {r}+h'(r)) g'(r)+\left(\l_{2,\a}(B(R); \gamma_h)-\frac {m-1} {r^2}\right) g(r)=0
\end{equation}
with boundary condition $g(0)=0$ and $g'(R)=-\a g(R)$. Here $\l_{2, \a}(B(R);\a)$ denotes the second Robin eigenvalue of \eqref{eq 2.2}. 
\end{proposition}
\begin{proof}
  Since the second eigenvalue $\l_{2, \a}(B(R);\gamma_h)$ of \eqref{eq 2.2} is  either the first eigenvalue $\mu_1$ of \eqref{eq 2.5}
 or the second eigenvalue $\tau_2$ of \eqref{eq 2.4},  it suffices to show
\begin{align*}
    \mu_1<\tau_2.
\end{align*}
Recall that an eigenfuction corresponding to the second eigenvalue of \eqref{eq 2.4} must change sign in $(0,R)$. Without loss of generality, assume $f(r)$ is positive in $(0, a)$ for some $a<R$ and $f(a)=0$. Then  $f'(a)<0$, and the same argument as in part (1) of Proposition \ref{prop 2.2} shows $f'(r)<0$ for $r\in(0, a)$.
Let  $f(r)$ be an eigenfunction of \eqref{eq 2.4}, i.e.
\begin{align}\label{eq 2.6}
 f'' + \left( \dfrac{m-1}{r} + h' \right)f' + \tau_2 f = 0  
\end{align}
for $r\in (0, R)$, with $f'(0)=0$ and $f'(R)+\a f(R)=0$.
Differentiating  \eqref{eq 2.6} gives
\begin{align}\label{eq 2.13}
    f'''+(\frac{m-1}{r}+h')f''+(\tau_{2}-\frac{m-1}{r^{2}}+h'')f'=0.
\end{align}
Set $\phi(r)=e^{h(r)}f'(r)$,  then 
\begin{align*}
    \phi'(r)=e^{h(r)}f''+h'\phi, 
\end{align*}
and
\begin{align*}
    \phi''(r)&=e^{h(r)}f'''+e^{h(r)}h'f''+h''\phi+h'\phi'\\
    &=e^{h}(-(\frac{m-1}{r}+h'(r))f''-(\tau_2-\frac{m-1}{r^{2}}+h'')f')+e^{h(r)}h'f''+h''\phi+h'\phi'\\
    &=e^{h}f''(-(\frac{m-1}{r}+h'(r))+h')-(\tau_{2}-\frac{m-1}{r^{2}}+h'')\phi+h''\phi+h'\phi'\\
    &=(\phi'(r)-h'\phi(r))(-\frac{m-1}{r})-(\tau_{2}-\frac{m-1}{r^{2}})\phi+h'\phi'\\
&=(-\frac{m-1}{r}+h')\phi'-(\tau_{2}-\frac{m-1}{r^{2}}-\frac{m-1}{r}h')\phi.
\end{align*}
Thus, \eqref{eq 2.13} implies
\begin{align}\label{eq 2.14}
    \phi''+(\frac{m-1}{r}-h')\phi'+(\tau_{2}-\frac{m-1}{r^{2}}-\frac{m-1}{r}h')\phi=0.
\end{align}
From  \eqref{eq 2.7}, we have
\begin{align*}
    0&=g'' + \left( \dfrac{m-1}{r} + h' \right) g'+ (\mu_{1}-\frac{m-1}{r^{2}}) g\\
    &=g''+(\frac{m-1}{r}-h')g'+(\mu_{1}-\frac{m-1}{r^{2}}-\frac{m-1}{r}h') g+2h'g'+\frac{m-1}{r}h'g.
\end{align*}
and the facts $g(r)>0$,  $h'(r)\ge 0$ and $g'(r)>0$ for $r\in (0,a)$, we obtain
\begin{align}\label{eq 2.15}
    0\geq  g''+(\frac{m-1}{r}-h^{'})g'+(\mu_{1}-\frac{m-1}{r^{2}}-\frac{m-1}{r}h') g.
\end{align}
Because $f'(r)<0$ for $r\in (0,a)$, we have $\phi(r)<0$ for $r\in(0, a)$. Multiplying \eqref{eq 2.15} by $\phi$  yields
\begin{align}\label{eq 2.16}
    0\leq \phi g''+\phi g'(\frac{m-1}{r}-h')+(\mu_{1}-\frac{m-1}{r^{2}}-\frac{m-1}{r}h') \phi g.
\end{align}
Subtracting equation \eqref{eq 2.14} multiplied by  $g$
from \eqref{eq 2.16} gives
\begin{align}\label{eq 2.17}
    0\leq \phi g''-g\phi''+(\frac{m-1}{r}-h')(g'\phi-g\phi')+(\mu_{1}-\tau_{2})g\phi.
\end{align}
Multiplying \eqref{eq 2.17} by $r^{m-1}e^{-h}$ and integrating over $(0,a)$, we obtain
\begin{align}\label{eq 2.18}
    0\leq \int_{0}^{a}(\phi g''-g\phi''+(\frac{m-1}{r}-h')(g'\phi-g\phi')+(\mu_{1}-\tau_{2})g\phi)r^{m-1}e^{-h(r)}dr.
\end{align}
Direct computation yields
\begin{align*}
   & \int_{0}^{a}(\phi g''-g\phi'')r^{m-1}e^{-h(r)}dr\\
    =&\int_{0}^{a}(\phi g'r^{m-1}e^{-h(r)})'-\phi'g'r^{m-1}e^{-h}-(\frac{m-1}{r}-h')\phi g'r^{m-1}e^{-h}dr\\
    &-\int_{0}^{a}(g\phi'r^{m-1}e^{-h(r)})'dr+\int_{0}^{a}g'\phi'r^{m-1}e^{-h}dr+\int_{0}^{a}(\frac{m-1}{r}-h')g\phi^{'}dr,
\end{align*}
and therefore
\begin{align}\label{eq 2.19}
\begin{split}
       &\int_{0}^{a}(\phi g''-\phi''g)r^{m-1}e^{-h(r)}dr\\
    =&(\phi g'-g\phi')r^{m-1}e^{-h}\Big|_{0}^{a}+\int_{0}^{a}(\frac{m-1}{r}-h')(g\phi'-\phi g')r^{m-1}e^{h}dr. 
\end{split}
\end{align}
Substituting \eqref{eq 2.19} into  \eqref{eq 2.18} yields 
\begin{align}\label{eq 2.20}
    0\leq (\phi g'-g\phi')r^{m-1}e^{-h}\Big|_{0}^{a}+\int_{0}^{a}(\mu_{1}-\tau_{2})g\phi r^{m-1}e^{-h(r)}dr.
\end{align}
Since $g(0)=0$ and $\phi(0)=e^{h(0)}f'(0)=0$, we have 
\begin{align*}
    (\phi g'-g\phi')r^{m-1}e^{-h}\Big|_{0}^{a}=\phi(a)g'(a)-g(a)\phi'(a).
\end{align*}
Recall that  $g(a)>0$, $g' (a)>0$ and $\phi(a)=e^{h(a)}f'(a)<0$. From \eqref{eq 2.6}, we obtain  
\begin{align*}
    f''(a)+(\frac{m-1}{a}+h'(a))f'(a)+\tau_{2}f(a)=0
\end{align*}
Since $f(a)=0$ and $f'(a)\leq 0$, it follows that
\begin{align*}
    f''(a)+h'(a)f'(a)=-\frac{m-1}{a}f'(a)> 0,
\end{align*}
and consequently
\begin{align*}
    \phi'(a)=e^{h(a)}(f''(a)+h'(a)f'(a))> 0.
\end{align*}
Hence
\begin{align}\label{eq 2.21}
    (\phi g'-g\phi')r^{m-1}e^{-h}\Big|_{0}^{a}= (\phi(a) g'(a)-g(a)\phi'(a))a^{m-1}e^{-h(a)}< 0.
\end{align}
Substituting  \eqref{eq 2.21} into  \eqref{eq 2.20} gives
\begin{align*}
    0< \int_{0}^{a}(\mu_{1}-\tau_{2})g\phi r^{m-1}e^{-h(r)}dr
\end{align*}
Since $\phi(r)<0$ for $r\in (0,a)$, we conclude $ \mu_{1}< \tau_{2}$. This completes  the proof of  Proposition  \ref{prop 2.3}.
\end{proof}

\begin{proposition} \label{prop 2.4}
Suppose $h'(r)\ge 0$, and $\a\ge -\sigma_1(B(R); \gamma_h)$. Then
\begin{align*}
  \lambda_{2,\alpha}(B(R); \gamma_h)\geq 0.  
\end{align*}
\end{proposition}
\begin{proof}
    Let $g(r)$ be as defined in Proposition \ref{prop 2.3}.  
    Since 
    \begin{align*}
        \int_{\partial B(R)}g(r)\frac{x_{i}}{r}e^{h(r)}dA=0, 1\leq i\leq m,
    \end{align*}
    the functions $u_{i}=g(r)\frac{x_{i}}{r}, 1\leq i\leq m,$ are admissible test functions  for $\sigma_{1}(B(R), \gamma_h)$. Therefore,
    \begin{align}\label{eq 2.22}
        \sum_{i=1}^{m}\int_{B(R)}|\nabla u_{i}|^{2}\, d\gamma_h\geq \sigma_1(B(R); \gamma_h) \sum_{i=1}^{m}\int_{\partial B(R)}|u_{i}|^{2}e^{h(r)}dA,
    \end{align}
    where  $dA$ is  the induced measure on $\p B$.
Using 
    \begin{align*}
        \sum_{i=1}^{m}|\nabla u_{i}|^{2}=g'(r)^{2}+\frac{m-1}{r^{2}}g(r)^{2}.
    \end{align*}
    and the fact $h(r)\le h(R)$ for $r\in[0, R]$, we estimate from \ref{eq 2.22} 
    \begin{align}\label{eq 2.23}
        0&\leq \sum_{i=1}^{m}\int_{B(R)}|\nabla u_{i}|^{2}e^{h(r)}dx- \sigma_1(B(R); \gamma_h)\sum_{i=1}^{m}\int_{\partial B(R)}|u_{i}|^{2}e^{h(r)}dA\nonumber\\
        &\leq \omega_{m-1}\int_{0}^{R}\left(g'(r)^{2}+\frac{m-1}{r^{2}}g^{2}\right)e^{h(r)}r^{m-1}dr+\alpha \omega_{m-1} g(R)^{2}e^{h(R)}R^{m-1},
    \end{align}
    where $\omega_{m-1}$  denotes the surface area of the unit sphere $\mathbb{S}^{m-1}$ in $\R^m$.
Using integration by parts and \eqref{eq 2.11}, we compute 
\begin{align*}
    &\int_{0}^{R}\left(g'(r)^{2}+\frac{m-1}{r^{2}}g^{2}\right)e^{h(r)}r^{m-1}dr\\
    =\quad &gg'e^{h(r)}r^{m-1}\Big|_{0}^{R}-\int_{0}^{R}g(g'e^{h(r)}r^{m-1})'dr+\int_{0}^{R}\frac{m-1}{r^{2}}g^{2}e^{h(r)}r^{m-1}dr\\
    =\quad &g(R)g'(R)e^{h(R)}R^{m-1}-\int_{0}^{R}g(g''+(\frac{m-1}{r}+h')g')e^{h(r)}r^{m-1}dr\\
    &+\int_{0}^{R}\frac{m-1}{r^{2}}g^{2}e^{h(r)}r^{m-1}dr\\
    =&g(R)g'(R)e^{h(R)}R^{m-1}+\int_{0}^{R}g^{2}(\lambda_{2,\alpha}(B(R);\gamma_h)-\frac{m-1}{r^{2}})e^{h(r)}r^{m-1}dr\\
    &+\int_{0}^{R}\frac{m-1}{r^{2}}g^{2}e^{h(r)}r^{m-1}dr.
\end{align*}
Substituting this equality into  inequality \eqref{eq 2.23} yields
  \begin{align*}
      0\leq& \lambda_{2,\alpha}(B(R);\gamma_h)\int_{0}^{R}g^{2}(r)e^{h(r)}r^{m-1}dr+g(R)e^{h(R)}R^{m-1}(g'(R)+\alpha g(R))\\
      =&\lambda_{2,\alpha}(B(R);\gamma_h)\int_{0}^{R}g^{2}(r)e^{h(r)}r^{m-1}\, dr,
  \end{align*}
which implies $\lambda_{2,\alpha}(B(R);\gamma_h)\geq 0$.
\end{proof}

\section{Proof of  Theorem \ref{thm 1}}\label{sect3}
In this section, we prove that round balls maximize the second Robin eigenvalue of \eqref{eq 1.1} among all symmetric bounded Lipschitz domains about the origin with the prescribed volume of $\gamma_h$. 

From now on, we assume $\Omega\subset \R^m$ is  a bounded Lipschitz domain symmetric about the origin, and $B\subset \R^m$ is the origin-centered round bound with the same $\gamma_h$ voulme as $\Omega$. Let $R$ be the radius of $B$. We extend the function $g$ defined in \eqref{eq 2.11} by
\begin{equation}\label{eq 3.1}
g(r)=g(R) e^{-\a(r-R)}\text{\quad for \quad} r\ge R.
\end{equation}
By definition, $g$ is continuously differentiable on $(0,\infty)$. If $\a\le 0$, then by Proposition \ref{prop 2.2}, $g$ is increasing on $(0,\infty)$. In  the sequel, $\sigma_1(B; \gamma_h)$ denotes the first nonzero Steklov eigenvalue of $B$ as defined in \eqref{eq 2.2}.

\begin{lemma}\label{lm 3.1}
Assume  $\a \in [-\sigma_1(B; \gamma_h),0]$. Define $F:[0,\infty)\to\mathbb{R}$ by
\begin{equation}\label{defH}
F(r):=g'(r)^2+\frac {m-1} {r^2}g^2(r) +2\a g(r)g'(r)+\a(\frac{m-1}{r}+h'(r))g^2(r)
\end{equation}
where $g(r)$ is defined by \eqref{eq 3.1}.
Then $F$ is monotonically decreasing on $(0,\infty)$.
\end{lemma}
\begin{proof}
{\bf Case 1.} $0<r\le R$. Direct calculation gives
\begin{align*}
    F'(r)&=2g'g''+2\frac{m-1}{r^{2}}gg'-2\frac{m-1}{r^{3}}g^{2}\\
    &+2\alpha g'(r)^{2}+2\alpha g g''+2\alpha (\frac{m-1}{r}+h')gg'+\alpha(h''-\frac{m-1}{r^{2}})g^{2}
\end{align*}
Using  \eqref{eq 2.11}, we obtain
\begin{align*}
    F'(r)&=2g'(g''+\frac{m-1}{r}g)-2\frac{m-1}{r^{3}}g^{2}\\
    &+2\alpha g(g''+(\frac{m-1}{r}+h')g')+2\alpha g'(r)^{2}+\alpha (h''-\frac{m-1}{r^{2}})g^{2}\\
    &=2g'(-(\frac{m-1}{r}+h')g'-(\lambda_2 -\frac{m-1}{r^{2}})g+\frac{m-1}{r^{2}}g)-2\frac{m-1}{r^{3}}g^{2}\\
    &+2\alpha g^{2}(\frac{m-1}{r^{2}}-\lambda_2)+\alpha (h''-\frac{m-1}{r^{2}})g^{2}+2\alpha (g')^{2}\\
    &=-2\frac{m-1}{r}(g')^{2}+4\frac{m-1}{r^{2}}g'g-2\frac{m-1}{r^{3}}g^{2}-2h'(r)g'(r)^{2}-2\lambda_2 gg'\\
    &+\alpha \frac{m-1}{r^{2}}g^{2}+2\alpha g'(r)^{2}+\alpha h''g^{2}-2\alpha \lambda_2 g^{2}\\
    &=-2\frac{m-1}{r}(g'-\frac{1}{r}g)^{2}-2\lambda_2 g(g'+\alpha g)-2h'g'(r)^{2}+\alpha \frac{m-1}{r^{2}}g^{2}+2\alpha (g')^{2}+\alpha h''g^{2},
\end{align*}
where we denote $\l_2=\l_{2, \a}(B; \gamma_h)$  for brevity. 
Using  $h'\ge 0$, $h''\ge 0$ and $\a\le 0$, we obtain 
\begin{align}\label{eq 3.3}
   F'(r)\leq -2\frac{m-1}{r}(g'-\frac{1}{r}g)^{2}-2\lambda g(g'+\alpha g).
\end{align}
By the assumption $\a\ge -\sigma_1(B; \gamma_h)$ and \eqref{eq 2.3}, we have
$$
\a\ge -\frac 1 R,
$$
Part (2) of Proposition \ref{prop 2.3} gives 
\begin{align}\label{eq 3.4}
    g'(r)+\a g(r)\ge 0 , \quad r\in [0, R].
\end{align}
Recall from Proposition \ref{prop 2.4} that $\l_2\ge 0$, then  from \eqref{eq 3.3} and \eqref{eq 3.4}, we conclude
\begin{align*}
    F'(r)\leq -2\frac{m-1}{r}(g'-\frac{1}{r}g)^{2},
\end{align*}
which proves that $F(r)$ is monotonically decreasing on $(0, R)$.

{\bf Case 2.} $r\ge R$. For $r\ge R$, 
\begin{align*}
    F(r)=(-\alpha^{2}+\frac{m-1}{r^{2}}+\alpha\frac{m-1}{r}+\alpha h^{'})g^{2}(r),
\end{align*}
and a direct computation yields
\begin{align*}
    F'(r)=(2\alpha^{3}+\alpha h''-2\alpha^{2}h'-2\alpha^{2}\frac{m-1}{r}-3\alpha \frac{m-1}{r^{2}}-2\frac{m-1}{r^{3}})g^{2}(r).
\end{align*}
Since   $\alpha\leq 0$, $h'\geq 0$ and $h^{''}\geq 0$, we obtain
\begin{align*}
    F'(r)\le&    \left(-2\alpha^{2}\frac{m-1}{r}-3\alpha \frac{m-1}{r^{2}}-2\frac{m-1}{r^{3}}\right) g(r)^2\\
    =&-\frac{m-1}{r}(2(\alpha+\frac{3}{4r})^{2}+\frac{7}{8}\frac{1}{r^{2}})g(r)^2<0,
\end{align*}
proving that $F(r)$ is monotonically decreasing on $(0, R)$.
\end{proof}
 The main idea in proving Theorem \ref{thm 1}  is to construct trial functions for $\l_{2, \a}$ (see \eqref{eq 1.3}) using techniques introduced by Weinberger. We first recall a monotonicity lemma for weighted symmetrization.
\begin{lemma}\label{Lemma 3.2}
    Let $\Omega\subset \R^m$ be a bounded set, and let $B$ be the origin-centered round ball with the same $\gamma_h$ volume as $\Omega$. If $f(r)$ is monotonically decreasing  on  $[0,+\infty)$, then 
    \begin{align*}
        \int_\Omega f(|x|)\, d\gamma_m\le    \int_B f(|x|)\, d\gamma_m.
    \end{align*}
    Equality holds if and only if $\Omega=B$.
\end{lemma}
\begin{proof}
  Lemma 3.2 follows directly from Hardy-Littlewood inequality;  see the proof of  inequality (4.24) in \cite[Page 213]{BCB16}.
\end{proof}

Now we now proceed to prove Theorem \ref{thm 1}.
\begin{proof}
 Let $\Omega$ be a bounded Lipschitz domain symmetric about the origin in $\R^m$, and let $B\subset \R^m$ be the origin-centered round ball of radius $R$ satisfying 
 $$
\int_B \, d\gamma_h =\int_\Omega \, d\gamma_h.
 $$
Define $g(r)$ for $r\in [0, +\infty)$ as in \eqref{eq 3.1}, and for each $1\le i\le m$ set
\begin{align*}
    v_i(x)=g(|x|)\frac{x_i}{|x|}, 
\end{align*}
where $x=(x_1, x_2, \cdots, x_m)\in \R^m$.

Since  both $\Omega$ and $d\gamma_h$ are symmetric about the origin, and $\l_{1, \a}(\Omega;\gamma_h)$ is simple,  the first Robin eigenfunction, denoted by $u_1(x)$, is also symmetric about the origin. Consequently,
     \begin{align}\label{eq 3.5}
        \int_\Omega v_i(x)u(x)   \, d\gamma_h=0, \quad i=1,2,\cdots, m.
    \end{align}
Thus, by \eqref{eq 3.5} and variational characterization \eqref{eq 1.3} of $\l_{2, \a}(\Omega; \gamma_h)$, each $v_i(x)$ is admissible;, i.e. for $1\le i\le m$,
    \begin{align}\label{eq 3.6}
      \l_{2, \a}(\Omega; \gamma_h)  \int_{\Omega} v_i(x)^2 \, d\gamma_h \leq \int_{\Omega} |\nabla v_i(x) |^2 \, d\gamma_h+\a\int_{\p \Omega} v_i(x)^2e^{h(|x|)}\, dA
    \end{align}
A direct computation shows
\begin{equation*} 
        |\nabla v_i(x)|^2 = g'(r)^2 \frac{x_i^2}{|x|^2} + g(r)^2 (\frac 1 {|x|^2}-\frac{x_i^2}{|x|^4}).
\end{equation*}  
Substituting this into  \eqref{eq 3.6} and summing over $i$ yields
    \begin{align}\label{eq 3.7}
          \l_{2, \a}(\Omega; \gamma_h)  \int_{\Omega} g(r)^2\, d\gamma_h
\leq \int_{\Omega}  \left(g'(r)^2+\frac{m-1}{r^2}g(r)^2\right) \, d\gamma_h+\a\int_{\p \Omega} g(x)^2e^{h(|x|)}\, dA.
    \end{align}
Using the divergence theorem, we estimate 
\begin{align}\label{eq 3.8}
    \int_{\partial \Omega}g(r)^2 e^{h(|x|)}\, dA&\geq \int_{\partial \Omega}g(r)^2\langle \nabla r, \nu \rangle e^{h(|x|)}\, dA=\int_{\Omega}\div (g^{2}e^{h(|x|)}\nabla r)dx\nonumber\\
    &=\int_{\Omega}\left(g^{2}\Delta r+\nabla g^{2}\cdot\nabla r+g^{2}\nabla h \cdot \nabla r\right)\, d\gamma_h\nonumber\\
    &=\int_{\Omega}\left((\frac{m-1}{r}+h'(r))g(r)^{2}+(g(r)^{2})'\right)\, d\gamma_h.
\end{align}    
Since $\a\le 0$, substituting  \eqref{eq 3.8} to \eqref{eq 3.7} gives
\begin{align}\label{eq 3.9}
\begin{split}
  &\l_{2, \a}(\Omega; \gamma_h)\int_{\Omega} g(r)^2\, d\gamma_h\\
  \le& \int_{\Omega} \left(g'(r)^2+\frac{m-1}{r^2} g(r)^2)+\a (\frac {m-1} r+h'(r))g(r)^2+2\a g(r)g'(r)\right) \, d\gamma_h\\
  =&\int_{\Omega} F(r)^2 \, d\gamma_h,
  \end{split}
\end{align}
where $F(x)$ is defined in Lemma \ref{lm 3.1}.
From Proposition \ref{prop 2.2}, $g(r)$ is monotonically increasing on $(0, \infty)$; hence, by Lemma \ref{Lemma 3.2}, 
\begin{align}\label{eq 3.10}
    \int_{\Omega}g^{2}\, d\gamma_h \geq \int_{B}g^{2}(r)\, d\gamma_h.
\end{align}
Moreover, by Lemma \ref{lm 3.1} and Lemma \ref{Lemma 3.2}, we also have  
\begin{align}\label{eq 3.11}
\int_{\Omega}F(r) \, d\gamma_h\leq \int_{B}F(r)d\gamma_h.
\end{align}
Assembling inequalities \eqref{eq 3.9}, \eqref{eq 3.10} and \eqref{eq 3.11}, we get
\begin{align}\label{eq 3.12}
      \l_{2, \a}(\Omega; \gamma_h)\le \frac{\int_{B}F(r)d\gamma_h}{\int_{B}g(r)^{2}d\gamma_h}.
\end{align}
Recall that the functions $g(r)\frac {x_i}{|x|}$ are the eigenfunctions corresponding to $\l_{2,\a}(B;\gamma_h)$. Hence
\begin{align*}
\l_{2,\a}(B;\gamma_h)=\frac{\int_{B} g'(r)^2+\frac{m-1}{r^2}g(r)^2\, d\gamma_h + \a\int_{\p B} g(r)^2 e^{h(|x|)}\, dA}{\int_{ B} g(r)^2\, d\gamma_h}.
\end{align*}
Furthermore,
\begin{align*}
\int_{\partial B} g^2(r)e^{h(|x|)}\, dA&=\int_{\partial B_R} \langle g(r)^2e^{h(|x|)}\nabla r, \nu\rangle\, dA\\
&=\int_{B} \operatorname{div}\left( g(r)^2 e^{h(|x|)} \nabla r\right)\, dx\\
&=\int_{B}2g(r)g'(r)+(\frac {m-1}{r}+h'(r))g(r)^2  \, d\gamma_h.
\end{align*}
Therefore,
\begin{align}\label{eq 3.13}
\l_{2,\a}(B;\gamma_h) =\frac{\int_{B} F(r)\, d\gamma_h}{\int_{B} g(r)^2\, d\gamma_h} .
\end{align}
From \eqref{eq 3.12} and \eqref{eq 3.13}, we deduce
\begin{align*}
       \l_{2, \a}(\Omega; \gamma_h)\le      \l_{2, \a}(B; \gamma_h),
\end{align*}
which establishes inequality \eqref{eq 1.6}.

If equality holds inequality \eqref{eq 1.6},
    then equality in  \eqref{eq 3.10} and \eqref{eq 3.11} forces $\Omega=B$.
     This completes the proof.
\end{proof}
Now we prove Corollary \ref{cor}.
\begin{proof}
    Since $\lambda_{2,0}(\Omega; \gamma_h)$ is the second (first nonzero) eigenvalue and
  and $\lambda_{2,-\sigma_{1}(B;\gamma_h)}(\Omega)\leq \lambda_{2,-\sigma_{1}(B;\gamma_h)}(B)=0$, there exists $\alpha_{0}\in [-\sigma_{1}(B;\gamma_h),0]$ such that  $\lambda_{2,\alpha_{0}}(\Omega;\gamma_h)=0$.
Let $u$ be an eigenfunction for $\lambda_{2,\alpha_{0}}(\Omega;\gamma_h)$. Then 
\begin{align}\label{eq 3.14}
    \begin{cases}
\Delta u+\nabla h\cdot \nabla u = 0 & \text{in } \Omega, \\
\frac{\partial u}{\partial \nu} = -\alpha_0 u & \text{on } \partial \Omega.
\end{cases}
\end{align}
Multiplying the first equation in \eqref{eq 3.14} by $e^{h}$ and integrating over $\Omega$ yields
\begin{align*}
    0=\int_{\Omega}(\Delta u+\nabla h\cdot \nabla u)\, d\gamma_h=\int_{\partial \Omega}\frac{\partial u}{\partial \nu}e^{h}dA=-\alpha_{0}\int_{\partial \Omega}u e^{h}dA.
\end{align*}
Since $\alpha_{0}\neq 0$, we obtain $\int_{\partial \Omega}u e^{h}dA$=0.
Recalling the definition
\begin{align*}
    \sigma_{1}(\Omega; \gamma_h):=\inf \{ \frac{\int_{\Omega}|\nabla \varphi|^{2}\, d\gamma_h}{\int_{\partial\Omega}\varphi^{2}e^{h}dA}: \varphi\in W^{1,2}(\Omega; \gamma_h)\setminus\{0\},\int_{\partial \Omega}\varphi e^{h}dA=0\},
\end{align*}
and noting that $u$ satisfies the orthogonality condition, we have
\begin{align*}
    \sigma_{1}(\Omega;\gamma_h)&\leq \frac{\int_{\Omega}|\nabla u|^{2}\, d\gamma_h}{\int_{\partial\Omega}u^{2}e^{h}dA}\\
    =&\frac{\int_{\partial\Omega}u\frac{\partial u}{\partial \nu} e^{h}dA-\int_{\Omega}u(\Delta u+\nabla h\cdot \nabla u)\, d\gamma_h}{\int_{\partial\Omega}u^{2}e^{h}dA}\\
    &=\frac{-\alpha_{0}\int_{\partial \Omega}u^{2}e^{h}dA}{\int_{\partial\Omega}u^{2}e^{h}dA}\\
    =&-\alpha_{0}
    \leq \sigma_{1}(B;\gamma_h).
\end{align*}
Thus, the desired inequality follows. 
\end{proof}

 \bibliographystyle{plain}
	\bibliography{ref}

\end{document}